\documentclass[10pt, a4paper]{amsart}
\usepackage{verbatim}
\usepackage{graphicx}

\newtheorem{definition}{Definition}
\newtheorem{rem}{Remark}
\newtheorem{lemma}{Lemma}
\newtheorem{theorem}{Theorem}
\newtheorem{trans}{Transformation}

\DeclareMathOperator{\So}{SO}

\DeclareMathOperator{\id}{id}

\title{Regularizations of non-euclidean polygons}
\author[D. Vartziotis]{Dimitris Vartziotis}
\address{Institute of Structural Analysis \& Antiseismic Research\newline National Technical University Athens (NTUA)\newline 15780 Athens, Greece\newline \emph{and}\newline TWT GmbH Science \& Innovation\newline Mathematical Research \& Services\newline Ernsthaldenstr. 17\newline70565 Stuttgart, Germany}
\email{dimitris.vartziotis@nikitec.gr}
\author[D. Bohnet]{Doris Bohnet}
\address{TWT GmbH Science \& Innovation\newline Mathematical Research \& Services\newline Ernsthaldenstr. 17\newline 70565 Stuttgart, Germany}
\email{doris.bohnet@twt-gmbh.de}

\date{}
\normalsize
\begin{document}
\begin{abstract}
We are interested in \emph{easy} geometric transformations which regularize $n$-polygons in the non-euclidean plane. A transformation is called \emph{easy} if it can be easily implemented into an algorithm. This article is motivated by preceding work on geometric transformations on euclidean polygons and possible applications for non-euclidean meshes. 
\end{abstract}
\maketitle
\section{Introduction}
\subsection{Historical remarks and motivation}
For a long time geometry was thought of as independent of our perception and as somehow divine. Kant is a famous example for this point of view as he gave our elementary geometric constructions a privileged role in his concept of cognition describing them as a priori to any experience. The geometry all the mathematicians and philosophers before 1830 referred to was the euclidean geometry. And it was a Copernican-like revolution not only for our concept of mathematics but also of epistemology that finally Gauss, Bolyai and Lobachevsky (\cite{L40}) did no longer try to prove the parallel axiom out of the other axioms of euclidean geometry, but described the consequences of negating the parallel axiom.\footnote{From the three mathematicians associated to the discovery of non-euclidean geometry it is only Lobachevsky who published his ideas 1829-30 in a complete scientific way while Gauss reported on his ideas only in private letters as did Bolyai.} There are two possible scenarios which arise from the negation of the parallel axiom: Given a line $l$ and a point $P$ not lying on this line we can either postulate that there is no line $l'$ parallel to $l$ and going through $P$ or that at least two lines $l'$ and $l''$ parallel to $l$ go through this point $P$ (see figures below). In the first case, we describe the \emph{elliptic geometry}, in the second the \emph{hyperbolic geometry}. This was the birth of a new concept of geometry and enabled the formalization of geometry. \\    

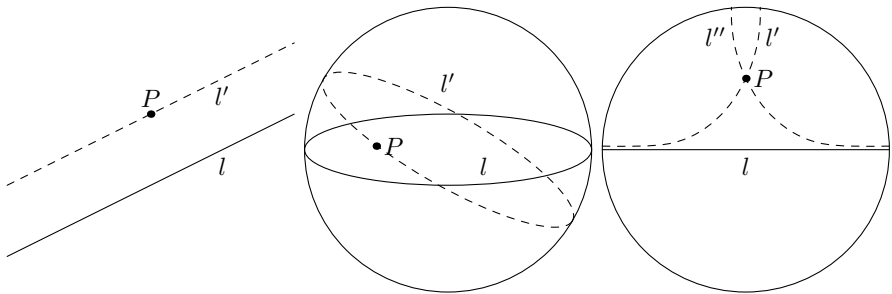
\begin{minipage}{0.3\textwidth}

\includegraphics[width=\textwidth]{./nichteuklidisch.mps}

\end{minipage}
\begin{minipage}{0.3\textwidth}
\includegraphics[width=\textwidth]{./nichteuklidisch1.mps}

\end{minipage}
\begin{minipage}{0.3\textwidth}
\includegraphics[width=\textwidth]{./nichteuklidisch2.mps}

\end{minipage}
\vspace{0.5cm}\newline
Instead of working with the axioms which define a geometry, Felix Klein perceived a geometry as a metric space where a group of transformations acts preserving certain characteristic invariants (cf \cite{K28}). In the case of the euclidean space the group of transformations is the group of isometries, i.e. all translations, rotations and reflections, and the respective invariant is the distance between two points.  \\
Looking at this piece of mathematical history we wondered if we could translate certain regularization methods described in the euclidean space into the non-euclidean setting. Regularizing polygons by simple geometric transformations plays an important role in industrial applications, and mathematical questions in this context often arise directly from practical problems: In engineering, any object is nowadays virtually given by a surface mesh of polygons (most often triangles) or a volume mesh of polyhedra (often tetrahedra or hexahedra). Theses meshes are on the one hand used to construct a new object, on the other hand to calculate certain properties (e.g. aerodynamical, acoustical or thermic properties) by using the mesh as discretization for the finite element method or finite volume method to solve the corresponding partial differential equations. \\ 
In both cases it is crucial to have a good mesh, i.e. a mesh of mostly regular polygons or polyhedra which adapts well to the real object. The regularization of meshes is therefore a standard process in engineering where the mostly used method is the following: You define a real function on the space of meshes which measure the quality of your mesh (e.g. a normalized volume function or an energy function). With optimization technics as Newton's method you calculate the maximum of this function which corresponds to a good mesh. This method is often not very efficient with regard to runtime. Additionally, it is often not clear if the solution really represents the global and not only a local maximum of the quality function. An alternative method (e.g. Laplace method or GETMe introduced and analyzed in \cite{VW09b},\cite{VW10}, \cite{VWS09} and \cite{VH12}) are direct variations of the polygons or polyhedra by elementary, easily computable geometric transformations. In \cite{VW12} (for polygons) or \cite{VH13} (for certain polyhedra) it is proved that the iterative method GETMe indeed converges to a good mesh, i.e. it maximizes an appropriate quality function. \\
These two aspects motivated us to consider easy geometric transformations which regularize a given polygon for non-euclidean geometries. At the moment there is no direct application for meshes of spherical or hyperbolic polygons but there are certain approaches, mainly in computer graphics, which use hyperbolic or spheric surfaces to model in a more elegant and adequate way a given object by incorporating its curvature. So it is not beside the point to think of future applications. \\
Returning to Klein's transformation groups we have to consider those transformations of euclidean polygons which are also adaptable to non-euclidean ones. For this reason, rotations will play a crucial role for us.\\
We would like to mention that there is considerable research going on about the structure of the space of hyperbolic polygons (see e.g. \cite{F11} and \cite{D13}), the generalization of certain euclidean properties or laws to the hyperbolic setting (see e.g. \cite{W10} or \cite{P12}) or about tesselations of hyperbolic surfaces (see e.g. \cite{D11} or \cite{LS12}). Our work is independent from these theoretical approaches and first and foremost driven by possible future applications.   
\subsection{Some classical transformations of euclidean polygons}
A classical transformation to regularize any euclidean triangle $\Delta=(z_0,z_1,z_2)$ with $z_i \in \mathbb{C}$ is the following associated to Napoleon.\footnote{There is no evidence for Napoleon's authorship but his name is historically associated with this construction. The first written proof of this easy statement is given by William Rutherford in the journal \emph{Lady's Diary} in 1825. There exists several different proofs using trigonometry or algebra which could be found e.g. on \emph{www.mathpages.com/home/kmath270/kmath270.htm}. A good survey of more recent generalizations of Napoleon's theorem is given in \cite{M96}.}
\begin{figure}
\begin{center}
\includegraphics[width=0.5\textwidth]{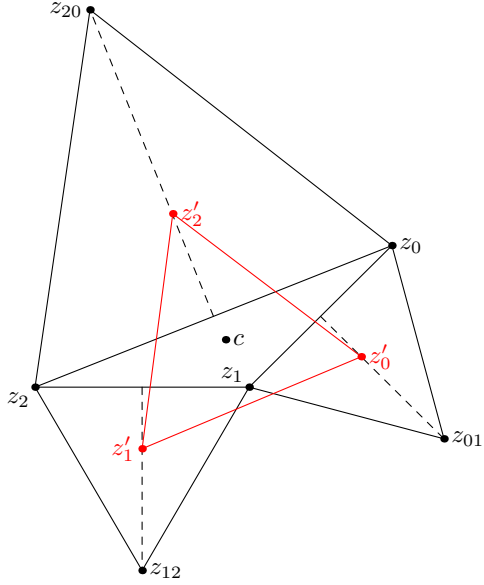}
\caption[Napoleon]{Napoleon's construction}
\end{center}
\label{fig:Napoleon}
\end{figure}
\begin{theorem}[Napoleon's Theorem]
If equilateral triangles are constructed on the sides of any triangle, either all outward, or all inward, the centres of those equilateral triangles themselves form an equilateral triangle.
\label{t.napoleon}
\end{theorem}
\begin{proof}
Let $z_0,z_1,z_2 \in \mathbb{C}$ denote the vertices of an arbitrary triangle in the plane. The centroid of this triangle is easily calculated to be $c=\frac{1}{3}(z_0+z_1+z_2)$. We will use the following auxiliary lemma: 
\begin{lemma}\label{l.auxiliary}
Let $z_0,z_1,z_2 \in \mathbb{C}$ be the vertices of a triangle $\Delta$. Then $\Delta$ is equilateral if and only if $$z_0z_1 + z_0z_2 + z_1z_2 = \frac{(z_0+z_1+z_2)^2}{3}.$$
\end{lemma}
\begin{proof}[Proof of Lemma~\ref{l.auxiliary}]
The proof is by calculation. The centroid of $\Delta$ is $c=\frac{1}{3}(z_0+z_1+z_2)$. Therefore the triangle is equilateral if and only if the distances $\left|z_j - c\right|=r$ are all equal for $j=0,1,2$ and the angles between them are $\frac{2\pi}{3}$. So (by rotating the whole triangle if necessary) we get $z_0-c=r$, $z_1-c=re^{\frac{2\pi i}{3}}$ and $z_2-c=re^{\frac{4\pi i}{3}}$. This implies that $$(z_1-c)(z_2-c)=(z_0-c)^2.$$
Substituting $c=\frac{1}{3}(z_0+z_1+z_2)$ and factorizing the equation we get 
\begin{align*}
z_0z_1 + z_1z_2 + z_0z_2 = \frac{1}{3}(z_1+z_2+z_3)^2
\end{align*}
\end{proof}
We construct an equilateral triangle over the side from $z_0$ to $z_1$ by using the formula of Lemma~\ref{l.auxiliary}: The third vertex $z_{01}$ is therefore
$$z_{01}=\frac{z_0+z_1}{2} + i\sqrt{3}\frac{z_0-z_1}{2}.$$
We construct the other two equilateral triangles in an analogous way and calculate their centroids $c_{0},c_1,c_2$, also by utilizing the formula of Lemma~\ref{l.auxiliary}, and we get $$c_0= \frac{z_0+z_1}{2} + \frac{i}{\sqrt{3}}\frac{z_0-z_1}{2}$$ for the centroid of the triangle $(z_0,z_{01},z_1)$. The other centroids are calculated in the same way. 
Now we only have to prove that $c_0,c_1,c_2$ form an equilateral triangle by checking that 
$$c_0c_1 + c_1c_2 + c_0c_2 = \frac{1}{3}(c_1+c_2+c_3)^2.$$
 
\end{proof} 
\begin{rem}   
It should be a priori clear that the formula for the centroid $c$ cannot be adapted one-to-one to the elliptic or hyperbolic plane while the characterization of a regular triangle by the equal distance of each vertex to $c$ and an angle of $\frac{2\pi}{3}$ between $z_j-c$ and $z_{j+1}-c$ will be helpful in all geometries. 
\end{rem}
\begin{figure}
\begin{center}
\includegraphics[width=0.5\textwidth]{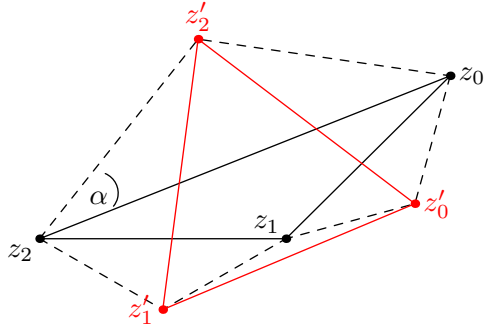}
\caption[Transformation2]{Transformation 2}
\end{center}
\label{fig:Transformation2}
\end{figure}
If we consider the construction, we remark that the vertices of the new equilateral triangle lie on the median line of each side of the original triangle. So a possible variation of Napoleon's triangle construction is the following: 
\begin{trans}\label{t.isosceles}
Construct over each side of an arbitrary triangle an isosceles triangle and connect their apices to a new triangle. 
\end{trans}  
There are a lot of possible ways to adapt this method in practice, e.g.
\begin{trans}\label{t.isosceles1}Fix an angle $\alpha < \frac{\pi}{2}$.
Construct over each side of an arbitrary triangle an isosceles triangle such that the two equal angles are $\alpha$. Connect their apices to a new triangle. 
\end{trans}
A concise discussion of this transformation can be found in \cite{G61}.  
Another specification of Transformation~\ref{t.isosceles} is the following which allows an especially simple mathematical description:
\begin{trans}\label{t.isosceles2}Let $\Delta=(z_0,z_1,z_2)$ be a triangle and $c$ its circumcenter.
Construct over each side of $\Delta$ an isosceles triangle such that the distance of the apices to the circumcenter $c$ is equal to the distance of the vertices $z_i$ to the circumcenter $c$. Connect then their apices to a new triangle. 
\end{trans}
\begin{figure}
\begin{center}
\includegraphics[width=0.5\textwidth]{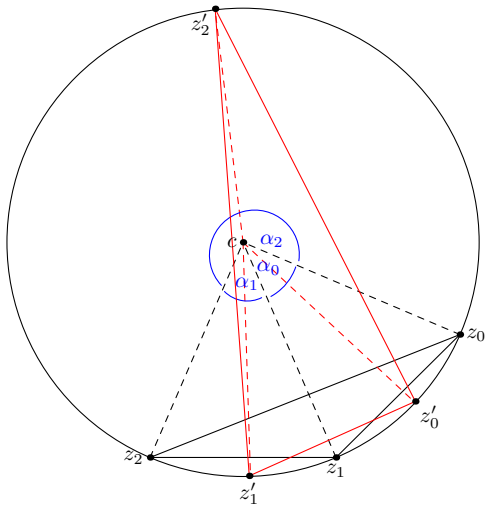}
\caption[Transformation3]{Transformation 3}
\end{center}
\label{fig:Transformation3}
\end{figure}
The advantage of Transformation~\ref{t.isosceles2} is obvious: Let $\Delta=(z_0,z_1,z_2)$ be a triangle and $c$ its circumcenter. It is then possible to rotate the point $z_0$ around $c$ such that it intersects the center of the side $z_0z_1$ and define the rotated $z_0$ as vertex of the new triangle. The triangle of $z_0,z_1$ and $c$ is isosceles because of $|z_0-c|=|z_1-c|=r$. So the bisecting line of the angle $\alpha$ between $z_0-c$ and $z_1-c$ coincides with the median line of the side $z_0z_1$. Consequently, our rotation is exactly a rotation of $z_0$ around $c$ of $\frac{\alpha}{2}$, i.e. $z_0'=z_0e^{\frac{\alpha}{2}i}$.\\ This method can be easily iterated, and we will show in Section ~\ref{s.spheric} that we finally change the vertices in a way that the angles between $z_j-c$ are equally $\frac{2\pi}{3}$ for $j=0,1,2$ which is equivalent to the fact that the triangle is equilateral as we have kept the distance $|z_j-c|=r$ throughout the transformation. \\ 
While Transformation~\ref{t.isosceles1} can be directly adapted to a $n$-polygon this is not possible for Transformation~\ref{t.isosceles2} as not every $n$-polygon for $n > 3$ has a circumscribed circle. An early study of this transformation represents the article \cite{P08}. 
\begin{trans}
Construct over each side of an arbitrary polygon an isosceles triangle and connect their apices to a new polygon.
\end{trans}
These transformations were the starting point for our considerations about non-euclidean polygons. So, it is time to leave the euclidean setting. As an introduction to the non-euclidean setting we start discussing the Napoleon's theorem for non-euclidean triangles before we prove in detail within the main part of this article the regularizing transformations for non-euclidean polygons.
\subsection{Napoleon's Theorem for non-euclidean triangles}
\subsubsection{Spheric triangles}
Before we start, we would like to mention the article \cite{T95} which deals with Napoleon-like properties of spherical triangles, but from a different point of view than we do. Consider a spheric triangle represented by three points $(z_0,z_1,z_2)$ on the sphere $\mathbb{S}^2$ and their connecting geodesics on the sphere. There exists a circumcenter $c$ and an (euclidean) radius $r$ such that all points $z_j$ lie on a circle of radius $r$ around $c$. If we understand the points $z_j$ as unit vectors in $\mathbb{R}^3$, we see that they together with the origin form a tetrahedron where the three triangle sides at the origin are isosceles triangles with side length $1$ equal to the radius of the sphere. Therefore, the angle at the origin between the vectors $z_j$ define uniquely the (euclidean) length of the opposite triangle side $z_jz_{j+1}$. Consequently, it is clear that if the angles between the vectors $z_j$ are all equal, that then the spheric triangle is equilateral. Alternatively, you can rotate the triangle sides around the vector $c$ by $\frac{2\pi}{3}$ mapping $z_j$ to $z_{j+1}$ for $j\in\mathbb{Z}_3$ proving that the angles between $z_j-c$ are equally $\frac{2\pi}{3}$. Hence, the definition of an equilateral euclidean triangle given above can be easily adapted to a spheric triangle. Hence, we can prove Napoleon's Theorem in the following way:
\begin{theorem}[Napoleon's Theorem on the sphere]
Let $\Delta$ be a triangle on the sphere $\mathbb{S}^2$. Then we obtain a regular triangle by constructing equilateral triangles on each side and connecting their centroids to the new triangle. 
\end{theorem}
\begin{proof}
Let the triangle $\Delta$ be defined by three unit vectors $z_0,z_1,z_2 \in \mathbb{R}^3$. Join these vectors to a tetrahedron. If we cut this tetrahedron with a plane $P$ parallel to the euclidean triangle $\Delta_e$ formed by $z_0,z_1,z_2$, then the intersection always gives us a triangle $\Delta'_e$ similar to $\Delta_e$. Moreover, the circumcenter vector $c$ intersects $\Delta'_e$ in its circumcenter (defined within the plane $P$). Choose one $\Delta'_e$ and construct equilateral (euclidean) triangles on each side such that they all lie in the plane $P$ parallel to $\Delta_e$. Take the circumcenters of these three triangles and connect them to a new triangle $\Delta'_{e,new}$ which is equilateral by Theorem~\ref{t.napoleon}. Taking the vertices $z'_{0,new},z'_{1,new}$ and $z'_{2,new}$ of this new triangle as vectors, we normalize them to unit vectors $z_{0,new},z_{1,new}$ and $z_{2,new}$ which define us a new spheric triangle $\Delta_{new}$ on the sphere. \\
We now show that $\Delta_{new}$ is equilateral. Note that the circumcenter vector $c$ of $\Delta$ is still the circumcenter vector of $\Delta_{new}$. If we rotate $z'_{j,new}$ around $c$ by $\frac{2\pi}{3}$ we map it onto $z'_{j+1,new}$ because $\Delta'_{e,new}$ is equilateral with circumcenter $c$. Clearly, we map also $z_{j,new}$ onto $z_{j+1,new}$ as they are just renormalizations of the first vectors. \\
What is left to prove is that the construction above is equivalent to constructing equilateral triangles on the side of our spheric triangle. Therefore consider again $\Delta'_e$ and the three equilateral triangles on its sides inside the plane $P$. Denote the three vectors of the new vertices by $z_{01}', z_{12}'$ and $z_{20}'$. If we normalize them to unit vectors $z_{01},z_{12}$ and $z_{20}$, we now have to show that the spheric triangle formed by $z_0,z_{01}$ and $z_{1}$ is equilateral (the other two triangles are analogously proven to be equilateral): Consider the vector $z_{0,new}$ which is a circumcenter of this triangle. If we rotate $z_0$ by $\frac{2\pi}{3}$ around $z_{0,new}$, it is mapped to $z_{01}$ because $z'_{0}$ rotated around $z_{0,new}'$ is mapped to $z_{01}'$, and the analogous statement is true for the other two points. So the spheric triangle $(z_0,z_{01},z_1)$ is an equilateral triangle on the side $z_0z_1$ of our original triangle and $z_{0,new}$ is its circumcenter. This proves our statement.    
\end{proof}

\begin{figure}
\begin{center}
\includegraphics[width=0.5\textwidth]{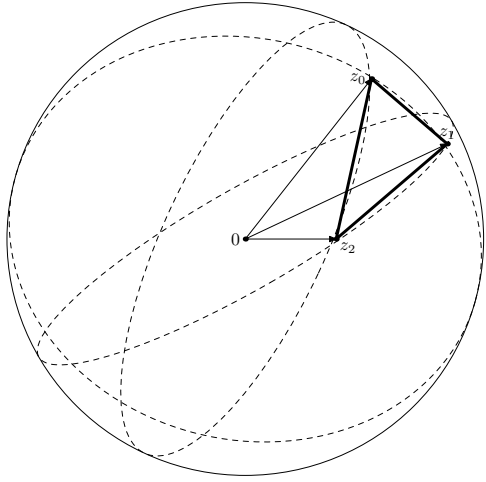}
\caption[spheric_triangle]{spheric triangle}
\end{center}
\label{fig:spheric_triangle}
\end{figure}

\subsubsection{Hyperbolic triangles}
Consider now a triangle $\Delta=(z_0,z_1,z_2)$ in the Poincar\'e disk $\mathbb{D}$. We call a hyperbolic triangle regular if all its angles are equal. This triangle is uniquely defined by three geodesics, arcs, which define the sides of the triangle. These geodesics meet the border $\partial \mathbb{D}$ orthogonally in six points $a_0,\dots,a_5$. If the inner angles of the triangles are equal the distances between the points are triplewise equal, i.e. of the form $(a,b,a,b,a,b)$ if we consider the border as $\mathbb{S}^1$ with the usual arc length as metric. This implies that $\Delta$ is centered at the origin $0$, and we can then conclude that the distance of $0$ to each of the vertices is equal because the rotation is an isometry for the hyperbolic disk. In this case it would be reasonable to call the origin the circumcenter of the triangle. \\
Given any triangle $\Delta$ on the Poincar\'e disk, we assume that its circumcenter is the origin. We map the points $z_0,z_1,z_2$ into $\mathbb{C}$ keeping the origin fixed and using polar coordinates. Recall that only the radius change under this map, but not the angle between the vectors. Now we make Napoleon's construction in $\mathbb{C}$ and obtain the new triangle $\Delta^e_{new}=(z^e_{0,new},z^e_{1,new},z^e_{2,new})$ which is regular in $\mathbb{C}$. In particular, it has still the origin as its circumcenter and the angles between its vertex vectors are equal. So, there exist $r$ and $\theta$ such that we have $z^e_{0,new}=re^{\theta i}, z^e_{1,new}=re^{\left(\frac{2\pi }{3}+\theta\right)i}$ and $z^e_{2,new}=re^{\left(\frac{4\pi i}{3}+\theta\right)i}$. If $r \geq 1$, we rescale the triangle. Mapping these points into the Poincar\'e disk gives the hyperbolic radius $\rho=-\ln\frac{1+r}{1-r}$ while the angle is kept. So we have a new hyperbolic triangle centered at the origin with coordinates $z_{j,new}=-\ln\frac{1+r}{1-r}e^{\left(\frac{2\pi}{3} + \theta\right)ij}$ for $j=0,1,2$. And this is certainly a regular triangle in the Poincar\'e disk.
While this construction provides us with a regular triangle in the Poincar\'e disk we still have to prove that this triangle coincides with the triangle we obtain by making Napoleon's construction in the Poincar\'e disk: Look at $z_0,z_1$ in polar coordinates in $\mathbb{C}$ and construct $z_{01}$ as new vertex of an equilateral triangle in $\mathbb{C}$. If we translate the centroid of this triangle in the origin and map it into $\mathbb{D}$ we obtain certainly a regular triangle as the angles are kept. But this euclidean translation is not an isometry in the hyperbolic plane.

\begin{figure}\begin{center}
\includegraphics[width=0.5\textwidth]{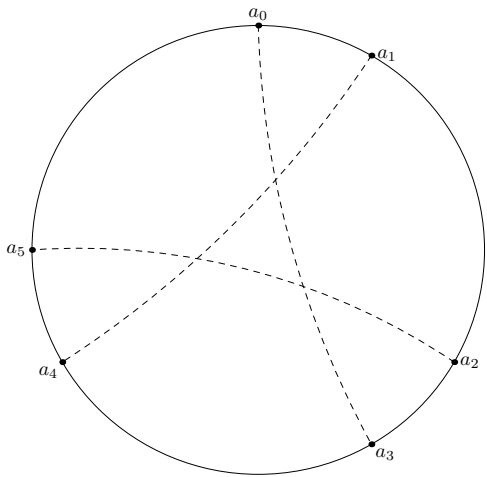}
\caption[hyp_triangle]{regular hyperbolic triangle}
\label{fig:hyp_triangle}

\end{center}

\end{figure}
\section{Regularization of spheric $n$-polygons}\label{s.spheric}
We consider spheric $n$-polygons and describe an easy algorithm which regularize them. Starting with the much easier case of a spheric triangle we explore later on the possibilities to generalize the employed method to arbitrary spheric $n$-polygons. 
First of all, we define what we call a regular spheric $n$-polygon. We consider the elliptic plane modelled as the $2$-sphere $\mathbb{S}^2$ embedded into $\mathbb{R}^3$, so points in the spheric plane correspond to unit vectors in $\mathbb{R}^3$. We write $v-w$ for vectors $v,w\in \mathbb{S}^2 \subset \mathbb{R}^3$ to denote the vector which join $v$ and $w$ within $\mathbb{R}^3$. Especially, we denote by $\left\|v\right\|, v\in \mathbb{R}^3$, the usual euclidean norm and not - if not otherwise specified - the length on the sphere.   
\begin{definition}
A \textbf{spheric n-polygon} is a $n$-tupel $P_n=(z_0, \dots,z_{n-1})$ of points $z_i \in \mathbb{S}^2$ on the sphere which denote the vertices of the polygon counter clockwisely counted. We call $P_n$ \textbf{regular} if there exists a matrix $A \in \So(3)$ such that $A^n=\id$ and $z_i=A^iz_0$ for $i=0,\dots n-1$. 
\end{definition}
\begin{rem}
Every matrix $A \in \So(3)$ is uniquely defined by a vector which describes its rotation axis and an angle of rotation. So changing the coordinate system (using the real Jordan decomposition for $A$) such that the rotation axis corresponds to the $z$-direction provides us with $A$ in the canonical form $$A= \begin{pmatrix}\cos (2\pi/n) & \sin (2\pi/n) &0\\-\sin(2\pi/n) & \cos(2\pi/n)&0\\0&0&1\end{pmatrix}.$$
For $n=3$ the rotation axis is just the middle point of the circumscribed circle of the triangle. 
\end{rem}
\subsection{Regularization of spheric triangles}

We start considering a spheric triangle $\Delta=(z_0,z_1,z_2)$. Denote by $D(\phi,v) \in \So(3)$ the rotation of $\phi$ around the axis $v$. Recall that every spheric triangle (exactly as a usual euclidean triangle) has a circumscribed circle on the sphere such that all of its vertices lie on this circle. But before we start we make clear what we mean by saying that a triangle converges to a regular triangle. Denote by $c \in \mathbb{S}^2$ the center of the circumscribed circle, i.e. the circumcenter, and by $\alpha_j$ the angle between the vectors $c-z_j$ and $c-z_{j+1}$ for $j \in \mathbb{Z}_3$. 
\begin{definition} Let the notations be as above. 
We say that a spheric $n$-triangle \textbf{converges to a regular spheric triangle} if the vector of its angles $(\alpha_0,\alpha_1,\alpha_2) \in \mathbb{R}^3$ converges to $(\frac{2\pi}{3}, \frac{2\pi}{3}, \frac{2\pi}{3})$ with respect to the canonical euclidean norm.
\end{definition}
\begin{rem}
As we keep the distance $\left\|c-z_j\right\|$ for $j=0,1,2$ throughout the transformation fixed, the criterion on the angles is sufficient to define regularity. 
\end{rem}
\subsubsection{First method}
We calculate the center $c \in \mathbb{S}^2$ of the circumscribed circle of $\Delta$.
We transform the triangle $\Delta$ by rotating the points $z_0,z_1$ and $z_2$ around $c$. The angle of rotation has to depend on the point to change the inner angle of the triangle. We choose the following transformation: Denote the start triangle by $\Delta^{(0)}=(z_0^{(0)},z_1^{(0)},z_2^{(0)})$. 
We measure the angle $\alpha_0^{(0)}$ between the vectors $c-z^{(0)}_0$ and $c-z^{(0)}_1$ and rotate $z^{(0)}_0$ by half of this angle (compare Figure~\ref{fig:Transformation3}), this means: 
$$z_0^{(1)}=D\left(c,\frac{\alpha^{(0)}_0}{2}\right)z_0^{(0)}.$$ 
The other points are transformed in the same way: 
\begin{align*}
z_1^{(1)}&=D\left(c,\frac{\alpha^{(0)}_1}{2}\right)z_1^{(0)},\\
z_2^{(1)}&=D\left(c,\frac{\alpha^{(0)}_2}{2}\right)z_2^{(0)},
\end{align*}
where $\alpha^{(0)}_1$ denotes the angle between the vectors $c-z_1^{(0)}$ and $c-z_2^{(0)}$ and $\alpha^{(0)}_2$ the angle between vectors $c-z_2^{(0)}$ and $c-z_0^{(0)}$. \\
It suffices to understand how the angles change in order to understand the mechanism of this algorithm. As the circumcenter $c$ and the distance from $c$ to the vertices are preserved by the transformation we can construct the triangle out of the three angles $\alpha_0,\alpha_1$ and $\alpha_2$. \\
We could understand $\alpha^{(0)}=(\alpha^{(0)}_0, \alpha^{(0)}_1, \alpha^{(0)}_2)$ as vector, and its transformation as a matrix: 
$$\alpha^{(1)}=A\alpha^{(0)}, \; \begin{pmatrix}\alpha^{(1)}_0\\ \alpha^{(1)}_1\\ \alpha^{(1)}_2\end{pmatrix}=\begin{pmatrix}\frac{1}{2} & \frac{1}{2} & 0 \\0 & \frac{1}{2} & \frac{1}{2}\\\frac{1}{2} & 0 & \frac{1}{2} \end{pmatrix}\begin{pmatrix}\alpha^{(0)}_0\\\alpha^{(0)}_1\\\alpha^{(0)}_2\end{pmatrix}$$

The matrix $A$ is a circulant matrix\footnote{A survey on circulant matrices is the book \cite{D79} by Davis. That circulant matrices can be useful to describe linear transformations of $n$-polygon is already folklore, see e.g. \cite{D40} or \cite{Da79}.}, and so the eigenvalues are easily deduced to be $\lambda_0=1, \lambda_1=\frac{1 + \sqrt{3}i}{4}$ and $\lambda_2=\overline{\lambda_1}$. Using the three eigenvectors $v_0,v_1,v_2$ as new coordinate system we can represent $A$ as $$A_{new}=\begin{pmatrix}1 & 0 & 0\\0 & \frac{1}{2} \cos(\frac{\pi}{3}) & \frac{1}{2} \sin(\frac{\pi}{3})\\0 & -\frac{1}{2}\sin(\frac{\pi}{3}) & \frac{1}{2} \cos(\frac{\pi}{3})\end{pmatrix}.$$ If we represent every angle vector $\alpha=(\alpha_0,\alpha_1,\alpha_2)$ (with $\alpha_0 + \alpha_1 + \alpha_2 =1$ corresponding to $2\pi$) within this new coordinate system we get $\alpha=a_0v_0 + a_1v_1 + a_2v_2$ and $\alpha_{new}=(a_0,a_1,a_2)$. It is important to note that $a_0 \neq 0$: 
\begin{lemma}\label{lemma_a0}
Let $\alpha=(\alpha_0,\alpha_1,\alpha_2) \in \mathbb{R}^3$ with $\alpha_0 + \alpha_1 + \alpha_2=1$. With the notations above we have $a_0 \neq 0$ for $\alpha=a_0v_0 + a_1v_1 +a_2v_2$. 
\end{lemma}
\begin{proof}
Assume that $a_0=0$, then we have $\alpha=a_1v_1 + a_2v_2$. Utilizing the hypothesis that $\sum_{j=0}^2 \alpha_j = 1$ we get $\sum_{j=0}^2(a_1 v_1^{(j)} + a_2v_2^{(j)})=1$. Recombining this sum we obtain $a_1\sum_{j=0}^2v_1^{(j)} + a_2 \sum_{j=0}^2v_2^{(j)}=1$, but $\sum_{j=0}^2v_1^{(j)}= \sum_{j=0}^2v_2^{(j)}=0$ contradicting the assumption. 
\end{proof}
In the direction of the eigenvector $v_0=(1,1,1)$ for the eigenvalue $1$ the transformation $A$ acts as the identity so that the angles do not change. This corresponds exactly to the case that $\alpha_{new}=(1,0,0)$, i.e. $\alpha_0=\alpha_1=\alpha_2=\frac{1}{3}$, and the triangle is already regular. \\
It should be remarked that we could not have angles $\alpha$ inside the plane spanned by the two complex eigenvectors $v_2=(-1+\sqrt{3}i, 2,-1-\sqrt{3}i)$ and $v_3=\overline{v_2}$ as the first component of the angle is always not zero as proved above in Lemma~\ref{lemma_a0}. \\
Denote the submatrix $$\frac{1}{2} R_{\frac{\pi}{3}}:=\begin{pmatrix} \frac{1}{2} \cos(\frac{\pi}{3}) & \frac{1}{2} \sin(\frac{\pi}{3})\\-\frac{1}{2}\sin(\frac{\pi}{3}) & \frac{1}{2} \cos(\frac{\pi}{3})\end{pmatrix}$$ as it is a rotation by $\frac{\pi}{3}$ and a contraction by $\frac{1}{2}$. Now we can prove that this algorithm regularizes an arbitrary spheric triangle. Denote by $\Delta^{(n)}$ the triangle with angles $\alpha^{(n)}$ constructed around $c$, i.e. the vertices lie on a circle around $c$ with radius $\left\|c-z_j^{(0)}\right\|$ on the sphere: 
\begin{theorem}Let the notations be as above.
For any spheric triangle $\Delta^{(0)}$ with associated angles $\alpha^{(0)}=(\alpha_0^{(0)}, \alpha_1^{(0)},\alpha_2^{(0)})$ the sequence $\Delta^{(n)}$ obtained from $\alpha^{(n)}=A^n\alpha^{(0)}$ converges to a regular triangle.
\end{theorem}
\begin{proof}
Consider an arbitrary spheric triangle $\Delta^{(0)}$ and the associated angle vector $\alpha^{(0)}$ represented as $\alpha^{(0)}_{new}=(a_0,a_1,a_2)$ with respect to the basis  $v_0,v_1,v_2$. Then we get 
$$\alpha^{(n)}_{new} =\begin{pmatrix}a_0\\ \frac{1}{2^n} R_{\frac{n\pi}{3}}\begin{pmatrix}a_1\\a_2\end{pmatrix} \end{pmatrix} \rightarrow_{n \rightarrow \infty} \begin{pmatrix}a_0 \\ 0 \\ 0 \end{pmatrix}=:\alpha^{(*)}_{new}.$$
So, with respect to the euclidean norm we prove the convergence
\begin{align*}
&\left\|\alpha^{(n)}-\alpha^{(*)}\right\| = \left\|\alpha^{(n)}_{new}-\alpha^{(*)}_{new}\right\|\\
=&\sqrt{(a_0-a_0)^2 + \left(\frac{1}{2}\right)^{2n}(a_1^2 + a_2^2)}\\
=&\frac{1}{2^n}\sqrt{a_1^2 + a_2^2}\quad \longrightarrow 0 \quad\mbox{for}\; n \rightarrow \infty.
\end{align*}
Recall that $\alpha^{(*)}=\alpha^{(*)}_{0,new}v_0+\alpha^{(*)}_{1,new}v_1+\alpha^{(*)}_{2,new}v_2$ and therefore $\alpha^{(*)}=(a_0,a_0,a_0)$. So it is $a_0=\frac{1}{3}$. 
Consequently, $\Delta^{(n)}$, the triangle corresponding to the angles $\alpha^{(n)}$, converges to a regular triangle (with respect to the norm of its angle vector in $\mathbb{R}^3$).  
\end{proof}

\subsubsection{Second method}
Let the notations be as above. 
Using half of the angle to rotate every vertex corresponds to the construction of an isosceles triangle over each side of the triangle and connecting the top of these newly constructed triangles to obtain a new triangle. An alternative method could be implemented by allowing other than isosceles triangles: Let $k \geq 2$ be an integer. Then consider the following transformation
\begin{align*}
z_0^{(1)}&=D\left(r,\frac{\alpha^{(0)}_0}{k}\right)z_0^{(0)},\\
z_1^{(1)}&=D\left(r,\frac{\alpha^{(0)}_1}{k}\right)z_1^{(0)},\\
z_2^{(1)}&=D\left(r,\frac{\alpha^{(0)}_2}{k}\right)z_2^{(0)}.
\end{align*}
If we look at the transformation of the angles $\alpha^{(0)}$ we get 
$$\alpha^{(1)}=A\alpha^{(0)}, \; \alpha^{(1)}=\begin{pmatrix}\frac{k-1}{k} & \frac{1}{k} & 0 \\ 0 & \frac{k-1}{k} & \frac{1}{k} \\ \frac{1}{k} & 0 & \frac{k-1}{k}\end{pmatrix}\alpha^{(0)}.$$
Analogously, the matrix $A$ is circulant, and the eigenvalues are therefore the following 
$$\lambda_1 = 1, \; \lambda_2=\frac{1}{2k}(2k - 3 + \sqrt{3}i), \; \lambda_3=\overline{\lambda_3}$$
with the corresponding eigenvectors 
$$v_1=(1,1,1), \; v_2=(-1-\sqrt{3}i,2, -1 + \sqrt{3}i), \; v_3=\overline{v_2}.$$
Writing $A$ with respect to the coordinate system of eigenvectors it has the form 
$$A_{new}=\begin{pmatrix}1 & 0 \\0 & \frac{\sqrt{k^2 - 3k + 3}}{k} R_{\phi(k)}\end{pmatrix}$$
where $R_{\phi(k)}$ notates a planar rotation by $\phi(k)=\arccos\left(\frac{(2k-3)}{2\sqrt{k^2-3k+3}}\right)$.
\begin{theorem}Let the notations be as above, $\Delta^{(0)}$ a spheric triangle and $k\geq 2$. Then $\Delta^{(n)}$ associated to the angles $\alpha^{(n)}=A^n\alpha^{(0)}$ converges to a regular triangle. 
\end{theorem}
\begin{proof}
Let $\Delta^{(0)}$ have the associated angles $\alpha^{(0)}=(\alpha_0^{(0)},\alpha^{(0)}_1,\alpha_2^{(0)})$. Let $\alpha_{new}=(a_0,a_1,a_2)$ be the representation of $\alpha$ with respect to the base of eigenvectors. We consider 
\begin{align*}
\left\|\alpha^{(n)}-\alpha^{(*)}\right\|=&\left\|\alpha^{(n)}_{new} - \alpha^{(*)}_{new}\right\|\\
=& \sqrt{(a_0 - a_0)^2 + \left(\frac{k^2-3k+3}{k}\right)^{2n}(a_1^2 + a_2^2)} \\
= &\left(\frac{k^2-3k+3}{k}\right)^n \sqrt{a_1^2 + a_2^2}\quad \mbox{for}\; k \geq 2\\
\leq& \left(\frac{k-1}{k}\right)^n \sqrt{a_1^2 + a_2^2} \quad \longrightarrow 0 \quad\mbox{for}\; n \rightarrow \infty.
\end{align*} 
Therefore, the angles $\alpha^{(n)}$ converge to $(\frac{1}{3},\frac{1}{3},\frac{1}{3})$ and hence, the triangle $\Delta^{(n)}$ to a regular triangle $\Delta^{(*)}$. 
\end{proof}
\begin{rem}Note that the velocity of the convergence depends on $k$ and the convergence becomes slower with a growing $k$ because for $k > 2$ we have 
$$\left(\frac{k-1}{k}\right)^n > \left(\frac{1}{2}\right)^n\quad \mbox{for all}\; n > 0.$$
\end{rem}
\subsection{Numerical Results}
We have implemented the method described above into Matlab and obtained the following convergence rates as an average of 20 experiments. To avoid long runtimes we interrupted the program after 20 iterations so the average iteration for $k=5$ has to be interpreted in the way that nearly half of the experiments would have needed more than 20 iterations to converge to a sufficiently regular triangle:\\
\begin{table}[h]
\centering
\begin{tabular}[h]{|c||c|c|c|c|}
\textbf{k} & 2 & 3 & 4 &5 \\ \hline
\textbf{\O \hspace{0.1cm} iterations} &7.5 & 8.3& 8.45 & 14.75\\ 
 \end{tabular}\vspace{0.2cm}
 \caption{Average of iterations necessary to obtain a regular triangle ($\pm 0.25$ error),  interruption 20 iterations} 
\end{table}

\subsection{Regularization of cyclic spheric $n$-polygons}\label{ss.cyclic}
For $n > 3$ an arbitrary $n$-polygon does not necessarily have a circumscribed circle. In order to be able to generalize our method above to $n$-polygons ($n > 3$) we assume therefore the following property: 
\begin{definition}
A $n$-polygon is called \textbf{cyclic} if it has a circumscribed circle. 
\end{definition}
Consider a cyclic spheric $n$-polygon denoted by $P_n=(z_0,\dots,z_{n-1})$. Construct the center $c$ of the circumscribed circle of this polygon and denote by $\alpha_j$ the angle between $c-z_j$ and $c-z_{j+1}$ for $j=0, \dots,n-1$. We have $\sum_{j=0}^{n-1}\alpha_j=2\pi$. We normalize this sum for simplicity to one. \\
As above we start with an arbitrary $n$-polygon $P_n^{(0)}$ and construct a new $n$-polygon $P_n^{(1)}$ by rotating every vertex $z_j^{(0)}$ by $\alpha_j/k$ around $c$ where $k \geq 2$ can be chosen arbitrarily.\\
The transformation for the angles is as following
$$\alpha_j^{(1)}=\frac{k-1}{k}\alpha_j^{(0)} + \frac{1}{k}\alpha_{j+1}^{(0)}, \quad j=0, \dots,n-1\; \mod n.$$
This transformation can be written as $\alpha^{(1)}=A\alpha^{(0)}$ where $A=(a_{ij})$ with $a_{ii}=\frac{k-1}{k}$ and $a_{ii+1}=\frac{1}{k}$. The matrix $A$ is therefore a circulant matrix, and the eigenvalues and eigenvectors are well known to be 
$$\lambda_j = \frac{k-1}{k} + \frac{1}{k}e^{\frac{2\pi ij}{n}}, \; v_j=(1, e^{\frac{2\pi ij}{n}}, \dots, e^{\frac{2(n-1)\pi ij}{n}}) \; \mbox{for}\; j=0, \dots, n-1.$$
We prove that the $n$-polygon $P^{(m)}$ constructed by taking $n$ vertices on the circle of radius $c-z_j^{(0)}$ around $c$ and keeping the angles $\alpha^{(m)}$ converges to a regular $n$-polygon. Analogously to the case of a triangle we say that a sequence of polygons $P^{(m)}$ \textbf{converges to a regular polygon} if its associated angles $\alpha^{(m)}\in\mathbb{R}^n$ converges to $(\frac{2\pi}{n},\dots, \frac{2\pi}{n}) \in\mathbb{R}^n$ within the canonical euclidean norm of $\mathbb{R}^n$. Before we start we need that the representation of $\alpha$ with respect to the coordinate system of eigenvectors is non-zero in the first entry. Otherwise the whole angle $\alpha^{(n)}$ would converge to zero. This is exactly the content of Lemma~\ref{lemma_a0} above which can be easily generalized to the case of a $n$-polygon: 
\begin{lemma}\label{lemma_a0_general}
Let $\alpha=(\alpha_0, \dots, \alpha_{n-1}) \in \mathbb{R}^n$ be a vector such that $\sum_{i=0}^{n-1}\alpha_i = 1$. With the notations above we have $a_0 \neq 0$ for $\alpha=\sum_{i=0}^{n-1}a_iv_i$. 
\end{lemma} 
\begin{proof}
Assume that $a_0=0$. Then we have $\alpha=\sum_{j=1}^{n_1}a_jv_j$. Utilizing that $\sum_{j=0}^{n-1}\alpha_j = 1$ we get $\sum_{k=0}^{n-1}(\sum_{j=1}^{n_1}a_jv^{(k)}_j)$. Changing the order of the summation this means $\sum_{j=1}^{n-1}a_j(\sum_{k=0}^{n-1}v^{(k)}_j)=1$. Consider for $j=1, \dots, n-1$ the sum $\sum_{k=0}^{n-1}v^{(k)}_j$. Applying the sum formula and using that $v^{(k)}_j=e^{\frac{2\pi ijk}{n}}=\left(e^{\frac{2\pi ij}{n}}\right)^k$ we get $\frac{1 - e^{2\pi ij}}{1- e^{\frac{2\pi ij}{n}}}$. But for $j \geq 1$ it is $e^{2\pi ij}=1$, so $\sum_{k=0}^{n-1}v^{(k)}_j=0$. This implies $\sum_{j=1}^{n-1}a_j(\sum_{k=0}^{n-1}v^{(k)}_j)=0$ contradicting the assumption.   
\end{proof} 
Now, we prove the following regularization theorem: 
\begin{theorem} Let the notations be as above, $P^{(0)}_n$ a cyclic spheric $n$-polygon. Then the sequence $P^{(m)}_n$ associated to $\alpha^{(m)}=A^m\alpha^{(0)}$ converges to a regular $n$-polygon. 
\end{theorem}
\begin{proof}
Let $P^{(0)}_n$ be an arbitrary cyclic spheric $n$-polygon with associated angle $\alpha^{(0)} \in \mathbb{R}^n$. Denote by $\alpha^{(0)}_{new}=(a_0,\dots,a_{n-1})$ the representation of $\alpha^{(0)}$ with respect to the eigenvectors of $A$. To show the convergence of the sequence $P^{(m)}$ we have to consider
$$\left\|\alpha^{(m)} - \alpha^{(*)}\right\|=\left\|A^m_{new}\alpha^{(0)}_{new} - \alpha^{(*)}_{new}\right\|$$
where $\alpha^{(*)}$ denotes the pointwise limit of $A^m\alpha^{(0)}$ for $m\rightarrow \infty$. We treat the cases of an odd number $n$ and an even number $n$ separately.\\
\textbf{Let $n$ be an odd number:} 
If $n$ is odd, we have only one real eigenvalue $\lambda_0=1$. As $A$ is a real matrix, we always find pairs of eigenvalues and corresponding eigenvectors which are pairwise conjugate, so we can reorder the eigenvalues into $\lambda_0, \lambda_1, \overline{\lambda}_1,\lambda_2, \overline{\lambda}_{2}, \dots, \lambda_{(n-1)/2}, \overline{\lambda}_{(n-1)/2}$, and we obtain a representation of $A$ as 
$$A_{new}=\begin{pmatrix} 1 & 0 & \dots & 0 \\0 & c_1 R_{\phi_1(k)}& 0 &\vdots\\0 & 0 & c_2 R_{\phi_2(k)} & \vdots \\ \vdots & \vdots & \ddots & 0 \\ 0 & 0 & 0& c_{(n-1)/2} R_{\phi_{(n-1)/2}(k)}\end{pmatrix}$$
where $c_j=c_j(k)= \frac{1}{k}\sqrt{k^2 - 2k + 2 + 2(k-1)\cos(2\pi j/n)}$ is the norm of $\lambda_j$ and $R_{\phi_j(k)}$ is a rotation by $\phi_j(k)=\arccos\left(\Re(\lambda_j)/c_j\right)$\footnote{More exactly, $\phi_j(k)=\arccos(\Re(\lambda_j)/c_j)$ for $\sin(2\pi j/n) \ge 0$, otherwise $\phi_j(k)=-\arccos(\Re(\lambda_j)/c_j)$}. Note that $c_j< 1$ for $j=1, \dots (n-1)/2$.\\
If we represent $\alpha$ corresponding to this new coordinate system of $v_0, v_1,v_{n-1}, v_2, v_{n-2}, \dots$ we get 
\begin{align*}
&\left\|\alpha^{(m)}-\alpha^{(*)}\right\|=\left\|A^m_{new}\alpha^{(0)}_{new}-\alpha^{(*)}\right\|\\
=&\sqrt{(a_0-a_0)^2 + \sum_{j=1}^{(n-1)/2}c_j^{2m}(a_{2j}^2 + a_{2j+1}^2)}\\
< & \left(\max_{j=1}^{n-1/2} c_j\right)^m \sqrt{\sum_{j=1}^{(n-1)/2}(a_{2j}^2 + a_{2j+1}^2)}\;\longrightarrow 0 \quad\mbox{for}\;m\rightarrow \infty\\
&\mbox{with}\; 0 <c_j < 1\; \mbox{for}\; j=1,\dots,\frac{n-1}{2}
\end{align*}
So, $P^{(m)}$ converges to a regular $n$-polygon. \\
\textbf{Let $n$ be an even number:}
If $n$ is even, we have two real eigenvalues $\lambda_0=1$ and $\lambda_{n/2}=\frac{k-1}{k} - \frac{1}{k}$ corresponding to the eigenvectors $v_0=(1, \dots,1)$ and $v_{n/2}=(1, -1, 1, \dots, -1)$. We reorder the eigenvalues as $\lambda_0, \lambda_{n/2}, \lambda_1, \overline{\lambda}_1, \dots, \lambda_{(n-2)/2}, \overline{\lambda}_{(n-2)/2}$. Corresponding to this order we get a representation of $A$ as 
$$A_{new}=\begin{pmatrix}\begin{matrix} 1 & 0 \\0 & \frac{k-2}{k}\end{matrix} & 0 & 0 &0\\ 0 &c_1 R_{\phi_1(k)}& 0& 0 \\  0 & 0 & \ddots& 0\\ 0 & 0 & 0& c_{(n-2)/2} R_{\phi_{(n-2)/2}(k)}\end{pmatrix}$$
with all notations as above.\\
It is easily seen that nevertheless the angles $\alpha^{(m)}_{new}$ converges to $(a_0,0,\dots,0)^T$ for $m \rightarrow \infty$ which correspond to the regular polygon.
\end{proof}
 
This shows that every cyclic $n$-polygon converges to a regular $n$-polygon. 
\begin{rem}
As above we shall remark that the convergence depends on the choice of $k\geq 2$, getting slower with a growing $k$. 
\end{rem}
\subsection{Regularization of spheric $n$-polygons}
Let $P_n=(z_0, \dots,z_{n-1})$ be a spheric $n$-polygon. Define a vector $m$ and a radius $r$ such that the euclidean distance (in $\mathbb{R}^3$) of every point $z_i$ to the circle of radius $r$ around $m$ is minimized. Project all $z_i$ onto this circle, i.e. draw the geodesic through $m$ and $z_i$ on the sphere and project $z_i$ onto $z'_i$, the intersection of the geodesic with the circle. We obtain a new $n$-polygon $P'_n=(z'_0,\dots,z'_{n-1})$ which is cyclic. Then we proceed as above in Subsection~\ref{ss.cyclic}.

\section{Regularization of hyperbolic $n$-polygons}\label{s.hyperbolic}
\subsection{Regularization of hyperbolic triangles}
Regarding the question of polygons and their regularization the most important difference of the hyperbolic plane to the spheric or euclidean plane is the fact that there exist infinitely many regular tessalations of the hyperbolic plane as the angle can become arbitrarily small. Consider the Poincar\'{e} disk $\mathbb{D}$ as model for the hyperbolic plane. Geodesics are semicircles orthogonal to the boundary $\mathbb{S}^1$ of $\mathbb{D}$. A triangle $\Delta=(z_0,z_1,z_2)$ is therefore defined by three semicircles $s_0,s_1,s_2$. 
\begin{definition}
A triangle is \textbf{regular} if all angles are equal. A sequence of hyperbolic triangles \textbf{converges to a regular triangle} iff there exists $0 \leq a < \frac{\pi}{3}$ such that the sequence of associated angles (seen as vector in $\mathbb{R}^3$) converges to a vector $(a,a,a)\in\mathbb{R}^3$ within the euclidean norm.  
\end{definition}
\begin{figure}

\begin{center}
\includegraphics[width=0.5\textwidth]{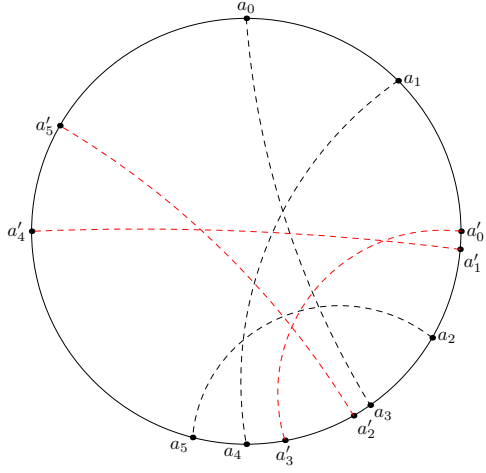}
\caption[hyp_triangle2]{hyperbolic triangle and its first iterate}
\label{fig:Dreiecke6}
\end{center}
\end{figure}
In the euclidean plane having equal angles is equivalent to the fact that every angle is $\frac{\pi}{3}$ as the sum of angles equals $\pi$. In the hyperbolic plane the sum of three inner angles of a triangle can obtain any angle smaller than $\pi$, even $0$, so we can prescribe any angle $\alpha$ smaller than $\frac{\pi}{3}$ and construct a regular triangle having three angles equal to $\alpha$. The case $\alpha=0$ corresponds to the ideal triangle whose geodesics meet at infinity, and its area is therefore infinite.\\
Consider now a triangle $\Delta$ defined by three geodesics $s_0,s_1$ and $s_2$. Each of these semicircles meets the boundary in two points $a_0,a_3$, $a_1,a_4$ and $a_2,a_5$, respectively. If the arc length of $a_0$ to $a_1$ written as $b_0:=d_{\mathbb{S}^1}(a_0,a_1)$ is equal to the arc lengths $b_2:=d_{\mathbb{S}^1}(a_2,a_3)$ and $b_4:=d_{\mathbb{S}^1}(a_4,a_5)$ (see Figure~\ref{fig:Dreiecke6}), then the triangle is \emph{regular}. We prove this easily: 
\begin{lemma}\label{l.regular}
Let $\Delta$ be a hyperbolic triangle defined by the intersections of the three geodesics $s_0,s_1,s_2$ in the Poincar\'e disk $\mathbb{D}$ and let $a_j,a_{j+3}$ be the points on the border $\partial \mathbb{D}$ where $s_j$ meets the border for $j=0,1,2$. Assume that the distances between the border points fulfill $d_{\mathbb{S}^1}(a_j,a_{j+1})=a$ for $j=0,2,4$ and $d_{\mathbb{S}^1}(a_j,a_{j+1})=b$ for $j=1,3,5$ (where $a_6=a_0$). Then the hyperbolic triangle defined by the intersection points of $s_0,s_1,s_2$ is regular, i.e. its inner angles are equal.    
\end{lemma}
\begin{proof}
Let $\Delta=(z_0,z_1,z_2) \subset \mathbb{D}$ be the triangle defined by $z_j \in s_j\cap s_{j+1}$ for $j=0,1,2$ and $s_3=s_0$. Denote the angle at $z_j$ by $\alpha_j$ for $j=0,1,2$. The points $a_j, j=0,\dots,5$ are on $\mathbb{S}^1$. If the distances between them are as assumed, it holds that the circle can be divided into three arcs of length $a+b$, so we have $a+b=\frac{2\pi}{3}$. Consequently, we get $a_{j+2}=e^{\frac{2 \pi i}{3}}a_j$ for $j \in \mathbb{Z}_6$. If we rotate $s_0$ and $s_1$ simultaneously by $\frac{2\pi}{3}$, then we map $a_0,a_3$ onto $a_2,a_5$ and $a_1,a_4$ onto $a_3,a_0$. As geodesics are uniquely defined by two points, we map indeed the geodesic $s_0$ onto $s_2$ and $s_1$ onto $s_0$. So the intersection point $z_0$ is mapped onto $z_2$, and as the angles are preserved by the simultaneous rotation of two geodesics, the angle $\alpha_0$ is equal to $\alpha_2$. If we rotate the geodesics $s_1$ and $s_2$, we map them onto $s_0$ and $s_1$, the point $z_1$ onto $z_0$, so we get $\alpha_1=\alpha_0$. This proves that the triangle $\Delta$ defined as above is a regular triangle.       
\end{proof}
Instead of transforming the triangle itself and dealing with hyperbolic distances, we would like to describe the triangle transformation in a mathematically much simpler way as a transformation of the endpoints $a_j\in \partial\mathbb{D}$ of the semicircles which define our triangle, more precisely as a transformation of the distances $b_j$ between these endpoints which we want to make triplewise equal.\\
The boundary circle $\mathbb{S}^1$ can be understood as $\mathbb{R}/\mathbb{Z}$ with additive structure, and we consider the end points of the semicircles $s_0,s_1,s_2$ defining a hyperbolic triangles as points $a_j \in \mathbb{R}/\mathbb{Z}$ for $j=0,\dots,5$. If we map the point $a_j^{(0)}$ to $a_j^{(1)}= a_j^{(0)} + \frac{1}{2}(a_{j+2}-a_j)$, this corresponds intuitively to the construction of an isosceles triangle over each side of our triangle and connecting the apices of these new three triangles. As we are interested in equalizing the distance between the points, it is more efficient to consider in fact the transformation onto the distances $b_j=\left|a_{j+1} - a_j\right|$ with $j\in \mathbb{Z}_6$ instead of the points.
It is easily calculated that the transformation does not change by this shift of our view point.\\
So, we have 
$$b^{(1)}=Ab^{(0)}, \; \begin{pmatrix}b_0^{(1)}\\ b_1^{(1)} \\ b_2^{(1)} \\ b_3^{(1)} \\ b_4^{(1)} \\ b_5^{(1)} \end{pmatrix}= \begin{pmatrix} \frac{1}{2} & 0 & \frac{1}{2} & 0 & 0 & 0 \\ 0 & \frac{1}{2} & 0 & \frac{1}{2} & 0 & 0 \\ 0 & 0 & \frac{1}{2} & 0 & \frac{1}{2} & 0 \\ 0 & 0 & 0 &\frac{1}{2} & 0 & \frac{1}{2} \\ \frac{1}{2} & 0 & 0 & 0 &  \frac{1}{2}&0 \\0 & \frac{1}{2} & 0 & 0 & 0 & \frac{1}{2} \end{pmatrix} \begin{pmatrix}b_0^{(0)}\\ b_1^{(0)} \\ b_2^{(0)} \\ b_3^{(0)} \\ b_4^{(0)} \\ b_5^{(0)} \end{pmatrix}.$$
This is a circulant matrix with the double eigenvalue $\lambda_0=\lambda_1=1$ and corresponding eigenvectors $v_0=(1,1,1,1,1,1)$ and $v_1=(1,-1,1,-1,1,-1)$. The other pairwise conjugate eigenvalues are $\lambda_2=\frac{3 + \sqrt{3}i}{4}$, $\lambda_3=\overline{\lambda_2}$ and $\lambda_4=\frac{1+\sqrt{3}i}{4}$, $\lambda_5=\overline{\lambda}_4$.  
We have the representation $$A_{new}=\begin{pmatrix} \begin{matrix}1 & 0 \\0 & 1 \end{matrix} & \begin{matrix} 0 & 0 \\0 & 0 \end{matrix} & \begin{matrix} 0 & 0 \\0 & 0 \end{matrix} \\ \begin{matrix} 0 & 0 \\0 & 0 \end{matrix} &\frac{3}{4}R_{\frac{\pi}{6}}& \begin{matrix} 0 & 0 \\0 & 0 \end{matrix}\\ \begin{matrix} 0 & 0 \\0 & 0 \end{matrix}& \begin{matrix} 0 & 0 \\0 & 0 \end{matrix}& \frac{1}{2}R_{\frac{\pi}{3}}\end{pmatrix}$$
with respect to the coordinate system of eigenvectors. \\
So, we get that $$A^n_{new}b^{(0)}_{new} \rightarrow_{n \rightarrow \infty} (b^{(0)}_{0,new},b^{(0)}_{1,new}, 0, 0, 0, 0)^T=:b^{(*)}_{new}.$$
By representing this limit vector $b^{(*)}_{new}$ within the canonical coordinates we obtain $$b^{(*)}=\begin{pmatrix}b^{(0)}_{0,new}+b^{(0)}_{1,new}\\b^{(0)}_{0,new}-b^{(0)}_{1,new}\\b^{(0)}_{0,new}+b^{(0)}_{1,new}\\b^{(0)}_{0,new}-b^{(0)}_{1,new}\\b^{(0)}_{0,new}+b^{(0)}_{1,new}\\b^{(0)}_{0,new}-b^{(0)}_{1,new}\end{pmatrix}$$
which shows exactly that the distances are triplewise equal, so we have the setting we were looking for in order to obtain a regular triangle.  
Now we can prove that our transformation provides us indeed with a regular triangle. We introduce the following notation: We calculate $b^{(n)}=A^nb^{(0)}$. We can always assume $a^{(0)}_0=0$, so we get points $a^{(n)}_0=0, a^{(n)}_1 =b^{(n)}_0$ and $a^{(n)}_i=a^{(n)}_{i-1} + b^{(n)}_{i-1}$ for $i=1,\dots,5$. For $i=0,1,2$ the points $a^{(n)}_i, a^{(n)}_{i+3}$ uniquely define a semicircle $s_i^{(n)}$ orthogonal at the border. We denote the triangle resulting from the intersection points of these three geodesics by $\Delta^{(n)}$. 
\begin{theorem}
With the notations above let $\Delta^{(0)} \subset \mathbb{D}$ be a triangle in the Poincar\'e disk. Then the triangle $\Delta^{(n)}$ converges to a regular triangle.
\end{theorem}
\begin{proof}
Let $\Delta^{(0)}$ be an arbitrary triangle in the Poincar\'e disk $\mathbb{D}^2$. It is then defined by three semicircles which meet the border $\partial \mathbb{D}$ orthogonally at six points. Denote these points clockwisely by $a_0, \dots, a_5$ and set $b^{(0)}_i=a_i-a_{i+1}$ for $i=0,\dots,5$ and $a_6=a_0$. Denote by $b^{(0)}_{new}=(b^{(0)}_{0,new},\dots,b^{(0)}_{5,new})$ the representation of the vector $b^{(0)}$ with respect to the eigen vectors of $A$ and by $b^{(*)}_{new}=(b^{(0)}_{0,new},b^{(0)}_{1,new}, 0, 0, 0, 0)^T$ the limit of $A^n_{new}b^{(0)}_{new}$ as above. Consequently, we have 
\begin{align*}
&\left\|b^{(n)}-b^{(*)}\right\|\\
=&\bigg(\left(b^{(0)}_{0,new}-b^{(0)}_{0,new}\right)^2 + \left(b^{(0)}_{1,new}-b^{(0)}_{1,new}\right)^2 + \\
&\left(\frac{3}{4}\right)^{2n}\left(\left(b^{(0)}_{2,new}\right)^2+ \left(b^{(0)}_{3,new}\right)^2\right) + \left(\frac{1}{2}\right)^{2n}\left(\left(b^{(0)}_{4,new}\right)^2 + \left(b^{(0)}_{5,new}\right)^2\right)\bigg)^{\frac{1}{2}}\\
< &\left(\frac{3}{4}\right)^{n}\bigg(\left(b^{(0)}_{2,new}\right)^2 + \left(b^{(0)}_{3,new}\right)^2 + \left(b^{(0)}_{4,new}\right)^2 + \left(b^{(0)}_{5,new}\right)^2\bigg)^{\frac{1}{2}}\\
&\longrightarrow 0 \quad \mbox{for}\; n \rightarrow \infty.
\end{align*}
Expressed in the canonical coordinates we have $b^{(*)}=(a,b,a,b,a,b)^T$ with $a=b^{(0)}_{0,new}+b^{(0)}_{1,new}$ and $b=b^{(0)}_{0,new}-b^{(0)}_{1,new}$. 
So the corresponding triangle $\Delta^{(*)}$ is with Lemma~\ref{l.regular} a regular triangle, and by definition the sequence $\Delta^{(n)}$ converges to $\Delta^{(*)}$.   
\end{proof}
\begin{rem}
Remark that the same argument as in Lemma~\ref{lemma_a0_general} can adapted to the case of the vector $b^{(0)}$ as the sum of its entries is exactly one, the length of the unit circle $\mathbb{R}/\mathbb{z}$. Consequently, we can conclude that $b_{0,new}^{(0)} \neq 0$ and as all coefficients are positiv, we get $a\neq 0$ for all triangles.\\
Therefore, the triangle $\Delta^{(*)}$ is \emph{ideal} if and only if $b$ vanish, i.e. $b_{0,new}^{(0)}=b_{0,new}^{(0)}$.
\end{rem}
\subsection{Regularization of hyperbolic $n$-polygons}
In opposite to the case of spheric $n$-polygon, we can directly adapt the method described above to hyperbolic $n$-polygons defined by the intersection of $n$-geodesics and therefore by $2n$ points on the border. 
\section{Conclusions and Outlook}
Looking at these transformations we would like to make some general remarks from a dynamical point of view about these transformations and to fortify theoretically why the transformations described above are exactly the good ones for our purpose. For simplicity we restrict ourselves to euclidean triangles, but as seen above, the arguments utilized were all borrowed from the euclidean setting. In \cite{D40} one can find a systematic algebraic discussion of linear transformations of polygons which covers our considerations for triangles under the assumption that every transformation is cyclic, i.e. one can permute the index of the points of the polygon without changing the transformation, but misses the dynamical focus. 
\subsection{General remarks about transformations on triangles}
Let $\Delta=(z_0,z_1,z_2)$ be an arbitrary euclidean triangle with $z_j \in \mathbb{C}$ and $c$ its circumcenter. We are looking for a transformation $T$ of $\Delta$ such that the sequence $T^n(\Delta)$ converges to a regular triangle. Let us suppose that the distance $\left|c-z_j\right|$ should be kept constant throughout the transformation.  This assumption is not as random as it might seem because it keeps the triangle centered at the same point and its area bounded. An immediate consequence of this assumption is that every transformation of the triangle is in fact a transformation of the angles $\alpha_j$ between the vectors $c-z_j$ and $c-z_{j+1}$. 
So, we should study transformations $$T: \mathbb{R}^3 \rightarrow \mathbb{R}^3, \; \alpha=(\alpha_0,\alpha_1,\alpha_2)\mapsto T(\alpha):=(T_0(\alpha),T_1(\alpha),T_2(\alpha)),$$
under the condition that $\sum_{j=0}^2T_j(\alpha)=\sum_{j=0}^2 \alpha_j = 1$ (corresponding to $2\pi$). Let $T$ be a smooth map. As the transformation $T$ should regularize our triangle we impose that 
$$ T^n(\alpha) \; \longrightarrow_{n \rightarrow \infty} \; \begin{pmatrix}\frac{1}{3}\\\frac{1}{3}\\\frac{1}{3}\end{pmatrix}=:\alpha^{(*)} \quad\mbox{for all}\; \alpha \in \mathbb{R}^3,$$
in particular, we want that a regular triangle is fixed under $T$, i.e.
$$T(\alpha^{(*)})=\alpha^{(*)}.$$
The two conditions stated above mean that the regular triangle $\alpha^{(*)}$ is an attracting fixed point of the transformation $T$. Consequently, the eigenvalues of the derivative $dT$ of $T$ in the fixed point has to be strictly smaller than $1$ as we have the following approximation for $\alpha$ sufficiently close to $\alpha^{(*)}$ 
$$T(\alpha) \; \sim \;\alpha^{(*)} + dT(\alpha^{(*)})(\alpha-\alpha^{(*)}).$$ To make life (and a future implementation) easier we restrict to linear transformations $T$: The fact that $\alpha^{(*)}$ is a fixed point is then equivalent that $\alpha^{(*)}$ is an eigenvector of $T$ to the single eigenvalue $1$. Taking the real Jordan decomposition we can directly conclude that $T$ is conjugate to one of the three following possibilities:
$$\begin{pmatrix}1 & 0 & 0 \\ 0 & a\cos(\phi) & a\sin(\phi)\\ 0 & -a\sin(\phi) & a\cos(\phi) \end{pmatrix} \quad \mbox{or}\quad \begin{pmatrix}1 & 0 & 0 \\ 0 & \lambda_1 & 0\\ 0 & 0 & \lambda_2 \end{pmatrix}\quad \mbox{or}\quad \begin{pmatrix}1 & 0 & 0 \\ 0 & \lambda_1 & 1\\ 0 & 0 & \lambda_1 \end{pmatrix}$$
where $ae^{i\phi}$ and $ae^{-i\phi}$ are the complex conjugate eigenvalues and $\lambda_1,\lambda_2$ denote the real eigenvalues. As $\alpha^{(*)}$ is supposed to be attracting we can conclude that $0 < a < 1$ and $0<\lambda_1,\lambda_2 < 1$. Zero eigenvalues must be excluded as they would permit that certain angles -lying in the corresponding eigenspace - stay unchanged.\\
As rotations are isometries not only in the euclidean, but also in the hyperbolic and elliptic setting, we have always dealt with the first matrix, a rotation matrix. This short dynamical discussion explains the sort of transformations we considered in Sections~\ref{s.spheric} and \ref{s.hyperbolic} above. 
\subsection{Outlook}
Regarding to possible generalizations and applications there are several natural questions one could pose:
\begin{description}
\item[Meshes] We considered the regularization of polygons because of its possible application for the regularization of polygon meshes. The next natural step would be to consider meshes of polygons and to generalize the transformations described above to this case. 
\item[Arbitrary surfaces] The surface of a general object one wants to mesh looks only locally like the euclidean plane. Also, the model for the hyperbolic or elliptic geometry cannot be used globally as a description of a surface. So, we would like to understand the meshing of a surface with positive or negative curvature using spheric or hyperbolic triangles.  
\end{description}

\bibliographystyle{plain}
\bibliography{bibfile}
\end{document}